\documentclass[10pt, leqno, draft]{amsart}
\usepackage{amsfonts}
\usepackage{ifthen}
\usepackage{amsthm}
\usepackage{amsmath,mathrsfs}
\usepackage{graphicx}
\usepackage{amscd,amssymb,amsthm}
\usepackage{color}
\usepackage{hyperref}
\usepackage{array,multirow,makecell}

\setcellgapes{1pt}
\makegapedcells
\newcolumntype{R}[1]{>{\raggedleft\arraybackslash }b{#1}}
\newcolumntype{L}[1]{>{\raggedright\arraybackslash }b{#1}}
\newcolumntype{C}[1]{>{\centering\arraybackslash }b{#1}}
\newcounter{minutes}
\divide\time by 60
\newcounter{hours}
\multiply\time by 60 \addtocounter{minutes}{-\time}

\newtheorem{theorem}{Theorem}
\newtheorem{lemma}{Lemma}

\newtheorem{corollary}{Corollary}
\newtheorem{remark}{Remark}
\newtheorem{example}{Example}

\title[Positivity of certain classes of functions related to the Fox H-functions ]{Positivity of certain classes of functions related to the Fox H-functions with applications}

\author[K. Mehrez]{Khaled Mehrez}
\address{Khaled Mehrez. D\'epartement de Math\'ematiques, Facult\'e des Sciences de Tunis, Universit\'e Tunis El Manar, Tunisia.}
\email{k.mehrez@yahoo.fr}



\begin{document}

\def\thefootnote{}
\footnotetext{ \texttt{File:~\jobname .tex,
          printed: \number\year-0\number\month-\number\day,
          \thehours.\ifnum\theminutes<10{0}\fi\theminutes}
} \makeatletter\def\thefootnote{\@arabic\c@footnote}\makeatother

\maketitle

\begin{abstract}
Our aim in this paper is to present a set of sufficient conditions  to be 
imposed on the parameters of the Fox H-functions which allow us to conclude 
that it is non-negative. As applications, some classes of completely monotonic and positive definite functions involving the H-function are established and various new facts regarding the Fox-Wright functions, including complete monotonicity, logarithmic completely monotonicity and monotonicity of ratios are considered.
\end{abstract}
\section{Introduction}

In 1961, Fox \cite{FOX} defined the $H$-function in his celebrated studies of symmetrical Fourier kernels as the Mellin-Barnes-type path integral
\begin{equation}
H_{q,p}^{n,m}[z]:=H_{q,p}^{n,m}\left[z\Big|^{({\bf b}_q, {\bf B}_q)}_{({\bf a}_p, {\bf A}_p)}\right]=H_{q,p}^{n,m}\left[z\Big|^{( b_j, B_j)_{1,q}}_{(a_i, A_i)_{1,p}}\right]=\frac{1}{2i\pi}\int_\mathcal{L}\mathcal{H}_{q,p}^{n,m}(s) z^{-s} ds,
\end{equation}
where 
$$\mathcal{H}_{q,p}^{n,m}(s)=\frac{\prod_{j=1}^n\Gamma (A_j s+a_j)\prod_{j=1}^m\Gamma(1-b_j-B_j s)}{\prod_{j=m+1}^q\Gamma (B_j s+b_j)\prod_{j=n+1}^p\Gamma(1-a_j-A_j s)}$$
is the Mellin transform of the Fox H-function $H_{q,p}^{n,m}[z]$ and $\mathcal{L}$ is the infinite contour in the complex plane which separates the poles
$$a_{il}=-\frac{a_i+l}{A_i},\;\;(i=1,...,n; l=0,1,2,...)$$
of the Gamma function $\Gamma(a_i+A_i s)$  to the left of  $\mathcal{L}$ and the poles
$$b_{jk}=\frac{1-b_j+k}{B_j},\;\;(j=1,...,m; k=0,1,2,...)$$
to the right of  $\mathcal{L}.$

The Fox H-function plays an important role in various branches of applied mathematics , many areas of theoretical physics, statistical distribution theory, and engineering sciences. The importance of the Fox H-function can be found in \cite{1, 11, 12, 13, MA, 29}, for instance.

Here, and in what follows, we use the Fox-Wright function ${}_p\Psi_q[.]?$ which is a generalization of hypergeometric function, is defined as follows \cite[p. 4, Eq. (2.4)]{Z}:
   \begin{equation}\label{11}
      {}_p\Psi_q\Big[\begin{array}{c} (a_1,A_1), \cdots, (a_p,A_p)\\(b_1,B_1), \cdots, (b_q,B_q) \end{array}
			   \Big|z \Big] 
			 = {}_p\Psi_q\Big[ \begin{array}{c}({\bf a}_p, {\bf A}_p)\\({\bf b}_q, {\bf B}_q)
			   \end{array} \Big|z \Big] 
       = \sum_{k \geq 0} \frac{\prod\limits_{l=1}^p\Gamma(a_l+kA_l)}
			   {\prod\limits_{l=1}^q\Gamma(b_l+kB_l)}\frac{z^k}{k!},
   \end{equation}
where $A_j \geq 0,\; j=1,\cdots, p; B_l\geq 0,\;\textrm{and}\;l=1,\cdots,q$. The convergence conditions 
and convergence radius of the series at the right-hand side of \eqref{11} immediately follow from the 
known asymptotic of the Euler Gamma--function. The defining series in \eqref{11} converges in the whole 
complex $z$-plane when
   \[ \Delta = 1+\sum_{j=1}^q B_j-\sum_{i=1}^p A_i>0. \]
If $\Delta=-1,$ then the series in (\ref{11}) converges for $|z|<\rho,$ and $|z|=\rho$ under the 
condition $\Re(\mu)>\frac12,$ (see \cite{16} for details), where 
   \[ \rho = \left(\prod_{i=1}^p A_i^{-A_i}\right)\left(\prod_{j=1}^q B_j^{B_j}\right),\quad 
			\mu  = \sum_{j=1}^q b_j-\sum_{k=1}^p a_k+\frac{p-q}2. \]
The Fox--Wright function extends the generalized hypergeometric function ${}_p F_q[z]$ which power series 
form reads 
   \[ {}_p F_q\Big[\begin{array}{c} {\bf a}_p \\ {\bf b}_q \end{array}\Big| z \Big]
			   = \sum_{k \geq 0}\frac{\prod\limits_{l=1}^p(a_l)_k}{\prod\limits_{l=1}^q(b_l)_k}\frac{z^k}{k!},\]
where, as usual, we make use of the Pochhammer symbol (or raising factorial) 
   \[ (\tau)_0 = 1; \quad (\tau)_k = \tau(\tau+1)\cdots(\tau+k-1) 
	             = \frac{\Gamma(\tau+k)}{\Gamma(\tau)},\qquad k\in\mathbb{N}. \]
In the special case $A_r=B_s=1$ the Fox--Wright function ${}_p\Psi_q[z]$ reduces (up to the 
multiplicative constant) to the generalized hypergeometric function 
   \[ {}_p\Psi_q\Big[\begin{array}{c} ({\bf a}_p,\mathbf 1)\\ ({\bf b}_q,\mathbf 1) \end{array} 
			    \Big| z \Big] 
			  = \frac{\Gamma(a_1)\cdots\Gamma(a_p)}{\Gamma(b_1)\cdots\Gamma(b_q)}\,
					{}_p F_q\Big[\begin{array}{c} {\bf a}_p \\ {\bf b}_q \end{array} \Big| z \Big]. \] 

The three-parameter Mittag-Leffler type function  $E_{\alpha,\beta}^{\gamma}(z)$ is defined by (see \cite{PA})
\begin{equation}\label{aki}
E_{\alpha,\beta}^{\gamma}(z):=\sum_{k=0}^\infty \frac{\Gamma(\gamma+k)}{\Gamma(\gamma)\Gamma(\beta+k\alpha)}\frac{z^k}{k!},\;\alpha>0, \beta,\gamma\in\mathbb{C}.
\end{equation}
However, we have that
\begin{equation}
E_{\alpha,\beta}^{\gamma}(z)=\frac{1}{\Gamma(\gamma)}{}_1\Psi_1\Big[_{(\beta, \alpha)}^{(\gamma, 1)}\Big|z \Big].
\end{equation}

The H-function contains, as special cases,  all of the functions which are expressible in terms of the
G-function. More importantly, it contains  the Fox-Wright generalized hypergeometric function defined in (\ref{11}), the generalized Mittag-Leffler functions, etc. For example, the function ${}_p\Psi_q[.]$ is one of these special case of H-function. By the definition (\ref{11}) it is easily extended to the complex plane as follows \cite[Eq. 1.31]{AA},
\begin{equation}\label{label}
{}_p\Psi_q\Big[_{({\bf b}_q,{\bf B}_q)}^{({\bf a}_p, {\bf A}_p)}\Big|z \Big]=H_{p,q+1}^{1,q}\left(-z\Big|_{(0,1),({\bf B}_q,1-{\bf b}_q)}^{({\bf A}_p,1-{\bf a}_p)}\right).
\end{equation}
The representation (\ref{label}) holds true only for positive values of the parameters $A_i$ and $B_j$. 

It is straightforward to verify that 
\begin{equation}
E_{\alpha,\beta}^{\gamma}(z)=\frac{1}{\Gamma(\gamma)}H_{1,2}^{1,1}\left[z\Big|^{(1-\gamma, 1)}_{(0, 1),(1-\beta,\alpha)}\right].
\end{equation}
 
The special case for which the  H-function reduces to the Meijer G-function is when $A_1=...=A_p=B_1=...=B_q=A,\;A>0.$ In this case,
\begin{equation}\label{,,,}
H_{q,p}^{m,n}\left(z\Big|^{({\bf B}_q,{\bf b}_q)}_{({\bf A}_p,{\bf a}_p)}\right)=\frac{1}{A}G_{p,q}^{m,n}\left(z^{1/A}\Big|^{{\bf b}_q}_{{\bf a}_p}\right).
\end{equation}
Additionally, when setting $A_i=B_j=1$ in (\ref{label}) (or $A=1$ in (\ref{,,,})), the H- and Fox-Wright functions  turn readily into the Meijer G-function.

For the reader's convenience, we first recall the definition of the completely monotone functions: A real valued function $f,$ defined on an interval $I,$ is called completely monotonic on $I,$ if $f$ has derivatives of all orders and satisfies
\begin{equation}
(-1)^nf^{(n)}(x)\geq0,\;n\in\mathbb{N}_0,\;x\in I.
\end{equation}

The  celebrated  Bernstein  Characterization  Theorem ( see, e.g., \cite{54})  that says that a function $f:(0,\infty)\rightarrow\mathbb{R}$ is completely monotone if and only if it can be represented as the Laplace transform of a non-negative measure (non-negative function or generalized function).

An infinitely differentiable function $f:I\longrightarrow[0,\infty)$ is called a Bernstein function on an interval $I,$ if $f^\prime$ is completely monotonic on $I.$ 

A positive function $f$ is said to be logarithmically completely monotonic on an interval $I$ if its logarithm $\log f$ satisfies
\begin{equation}
(-1)^n(\log f)^{(n)}(x)\geq0,\;n\in\mathbb{N},\;\;\textrm{and}\;\;x\in I.
\end{equation}
In \cite[Theorem 1.1]{Berg} and \cite[Theorem 4]{Y}, it was found and verified once again that a logarithmically completely monotonic function must be completely monotonic, but not conversely.

A continuous function $f:\mathbb{R}^d\longrightarrow\mathbb{C}$ is called positive definite function, if for all $N\in\mathbb{N},$ all sets of pairwise distinct centers $X=\left\{x_1,...,x_N\right\}\subseteq\mathbb{R}^d$ and $z=\left\{\xi_1,...,\xi_N\right\}\subset \mathbb{C}^N,$ the quadratic form
$$\sum_{j=1}^N\sum_{k=1}^N\xi_j\bar{\xi_k}f(x_j-x_k)$$
is non-negative.

The connection between positive definite radial and completely monotone functions, which was first pointed out by Schoenberg in 1938 (see \cite[Theorem 7.1]{W}), that is, a function $f$ is completely monotonic on $(0,\infty)$ if and only if $f(\left\|.\right\|^2)$ is positive definite on every $\mathbb{R}^d.$ For further details, we refer to Wendland \cite{W}.

Our aim in this investigation is to give sufficient conditions so that some classes of functions related to the Fox H-functions are non-negative functions. Applying this results, completely monotonicity and positive definite properties for some classes of functions associated to the H-function are derived and various new facts regarding the Fox-Wright functions,  including complete monotonicity, logarithmic completely monotonicity and monotonicity of ratios are established.

\section{Useful Lemmas}

In the proof of the main results we will need the following  lemmas, which we collect in this section. The first Lemma is about some properties for the Fox H-functions, see \cite{MM} for more details. 

\begin{lemma}\label{l1}Properties of Fox H-functions:\\
\noindent Property 1.
$$H_{q,p}^{n,m}\left[z^{-1}\Big|^{({\bf b}_q, {\bf B}_q)}_{({\bf a}_p, {\bf A}_p)}\right]=H_{p,q}^{m,n}\left[z\Big|_{(1-{\bf b}_q, {\bf B}_q)}^{(1-{\bf a}_p, {\bf A}_p)}\right].$$
\noindent Property 2.
$$ H_{q,p}^{n,m}\left[z\Big|^{({\bf b}_q, {\bf B}_q)}_{({\bf a}_p, {\bf A}_p)}\right]=k H_{q,p}^{n,m}\left[z^k\Big|^{({\bf b}_q, k{\bf B}_q)}_{({\bf a}_p, k{\bf A}_p)}\right], k>0.
$$
\noindent Property 3.
$$ H_{q,p}^{n,m}\left[z\Big|^{\;\;\;\;\;({\bf b}_q, {\bf B}_q)}_{({\bf a}_{p-1}, {\bf A}_{p-1}),(b_1,B_1)}\right]=H_{q-1,p-1}^{n-1,m}\left[z\Big|^{({\bf b}_j, {\bf B}_j)_{2,q}}_{({\bf a}_{i}, {\bf A}_{i})_{1,p-1}}\right].$$
\noindent Property 4. For $\sigma\in\mathbb{C},$ the following relation holds
$$ z^\sigma H_{q,p}^{n,m}\left[z\Big|^{({\bf b}_q, {\bf B}_q)}_{({\bf a}_p, {\bf A}_p)}\right]=H_{q,p}^{n,m}\left[z\Big|^{({\bf b}_q+\sigma{\bf B}_q, {\bf B}_q)}_{({\bf a}_p+\sigma{\bf BA}_q, {\bf A}_p)}\right].$$
\end{lemma}

As Lemma \ref{l12} indicates, the Hankel transform plays a central role in our investigation. We will need the following result for the Hankel transform for the H-function, see \cite[Formula (2.45), p. 57]{AA}.
\begin{lemma}\label{l12}
 We introduce the following parameters:
\begin{equation}\label{TH}
\begin{split}
C&=\sum_{j=1}^pA_j-\sum_{j=1}^q B_j\\
D&=\sum_{j=1}^m B_j-\sum_{j=m+1}^q B_j+\sum_{j=1}^n A_j-\sum_{j=n+1}^p A_j.
\end{split}
\end{equation}
Suppose that $D>0$ or $D=C=0$ and $\Re(\mu)<-1.$ Let $\nu, \rho\in\mathbb{C}$ such that
$$\Re(\rho)+\Re(\nu)+\min_{1\leq j\leq n}\left[\frac{\Re(a_j)}{A_j}\right]>-1,$$
and
$$\Re(\rho)+\min_{1\leq j\leq m}\left[\frac{1}{B_j}-\frac{\Re(b_j)}{B_j}\right]<\frac{3}{2},$$
Then for $x,b>0,$ we have
\begin{equation}\label{BBBB}\int_0^\infty r^{\rho-1} J_\nu(xr) H_{q,p}^{n,m}\left[br^\sigma\Big|^{({\bf b}_q, {\bf B}_q)}_{({\bf a}_{p}, {\bf A}_{p})}\right]dr=\frac{2^{\rho-1}}{x^\rho}H_{q+2,p}^{n,m+1}\left[b\left(\frac{2}{x}\right)^\sigma\Big|^{(1-\frac{\rho+\nu}{2},\frac{\sigma}{2}),({\bf b}_q, {\bf B}_q),(1-\frac{\rho-\nu}{2},\frac{\sigma}{2})}_{\;\;\;\;\;\;({\bf a}_{p}, {\bf A}_{p})}\right].
\end{equation}
\end{lemma}

The asymptotic expansion of H-function listed below, see \cite[Theorem 1.2, (i) and (ii)]{AA}.

\begin{lemma}
There holds the following results:\\
\noindent {\rm (i).} If $C\geq0$ or $C<0,\; \alpha>0, |\arg(z)|<\frac{1}{2}\pi D,$  the H-function has either the asymptotic expansion at zero given by
\begin{equation}\label{asy1}
H_{q,p}^{n,m}[z]=\mathcal{O}(z^c),\;|z|\rightarrow 0,
\end{equation}
where,
\begin{equation}
c=\min_{1\leq j\leq n}\left[\frac{\Re(a_j)}{A_j}\right].
\end{equation}
\noindent {\rm (ii).} If $C\leq0$ or $C>0,\;  \alpha>0,$ then the H-function has either the asymptotic expansion at infinity given by
\begin{equation}\label{asy2}
H_{q,p}^{n,m}[z]=\mathcal{O}(z^d),\;|z|\rightarrow \infty,
\end{equation}
here,
\begin{equation}
d=\min_{1\leq j\leq m}\left[\frac{\Re(b_j)-1}{B_j}\right].
\end{equation}
\end{lemma}

We will need the following Lemma for the Laplace transform of the H-function, for more details, see \cite[Eq. (2.19)]{AA}.

\begin{lemma}\label{Laplace} Let $D>0$ or $D=0$ and $C\geq0$ defined by (\ref{TH}). Suppose also that 
$$\min_{1\leq j\leq n}\left[\frac{\Re(a_j)}{A_j}\right]+1>0.$$
Then, the Laplace transform of the H-function exists, and there holds the relation
\begin{equation}\label{laplace}
\mathcal{L}\left\{H_{q,p}^{n,m}\left[z\Big|^{({\bf b}_q, {\bf B}_q)}_{({\bf a}_p, {\bf A}_p)}\right]; s\right\}=\frac{1}{s}H_{q+1,p}^{n,m+1}\left[\frac{1}{s}\Big|^{(0,1),({\bf b}_q, {\bf B}_q)}_{\;\;({\bf a}_p, {\bf A}_p)}\right],
\end{equation}
for $s\in\mathbb{C}$ with $\Re(s)>0.$
\end{lemma}

In  the following three lemmas we present some integral representations for the Fox-Wright functions, which plays a crucial role in the proof of some Theorems given in Section 3, see \cite{KMJMAA, KMCRM} for a proofs.

\begin{lemma}\label{T71} Suppose that $\mu>0,\;\gamma=\displaystyle\min_{1\leq j\leq p}(a_j/A_j)\geq1,\;\;\textrm{and}\;\;\sum_{j=1}^pA_j=\sum_{k=1}^qB_k.$
Then, the following integral representation
\begin{equation}\label{fr1}
{}_p\Psi_q\Big[_{({\bf b}_q, {\bf B}_q)}^{({\bf a}_p, {\bf A}_p)}\Big|z\Big]=\int_0^\rho e^{zt}H_{q,p}^{p,0}\left(t\Big|^{({\bf b}_q, {\bf B}_q)}_{({\bf a}_p, {\bf A}_p)}\right)\frac{dt}{t},\;\;(z\in\mathbb{R}),
\end{equation}
hold true. Moreover,  the function 
$$z\mapsto{}_p\Psi_q\Big[_{({\bf b}_q, {\bf B}_q)}^{({\bf a}_p, {\bf A}_p)}\Big|-z\Big]$$
is completely monotonic on $(0,\infty),$ if and only if, the function $H_{q,p}^{p,0}(z)$ is non-negative on $(0,\rho).$
\end{lemma}
\begin{lemma}\label{l3} Assume that the assumptions of Lemma \ref{T71} hold. Suppose also that the function $H_{q,p}^{p,0}[.]$ is non-negative. Then, the following Stieltjes transform holds true: 
\begin{equation}
{}_{p+1}\Psi_q\left[^{(\sigma,1),\; ({\bf a}_p, {\bf A}_p)}_{\;\;\;\;({\bf b}_q, {\bf B}_q)}\Big|-z\right]=\int_0^\rho \frac{d\mu(t)}{(1+tz)^\sigma},
\end{equation}
where
\begin{equation}
d\mu(t)=H_{q,p}^{p,0}\left(t\Big|^{({\bf b}_q, {\bf B}_q)}_{({\bf a}_p, {\bf A}_p)}\right)\frac{dt}{t}.
\end{equation}
Moreover, the function 
$$z\mapsto  {}_{p+1}\Psi_q\left[^{(\sigma,1),\;({\bf a}_p, {\bf A}_p)}_{\;\;\;\;({\bf b}_q, {\bf B}_q)}\Big|-z\right]$$
is logarithmically completely monotonic on $(0,1).$
\end{lemma}
\begin{lemma}\label{l4}{\rm \cite[Corollary 1]{KMCRM}} Suppose that $\mu=0,\;\gamma\geq1$ and $\displaystyle{\sum_{i=1}^p  A_i=\sum_{j=1}^q  B_j}.$ Then, the Fox-Wright function ${}_p\Psi_q[.]$ possesses the following integral representation
\begin{equation}\label{E1}
{}_p\Psi_q\Big[_{({\bf b}_q, {\bf B}_q)}^{({\bf a}_p, {\bf A}_p)}\Big|-z\Big]-\eta e^{-\rho z}=\int_0^\rho e^{-zt} d\mu(t),\;z\in\mathbb{R}
\end{equation}
where
\begin{equation}
\eta=(2\pi)^{\frac{p-q}{2}}\prod_{i=1}^p A_i^{a_i-\frac{1}{2}}\prod_{j=1}^q B_j^{\frac{1}{2}-b_j}.
\end{equation}
Moreover, if the function $H_{q,p}^{p,0}[.]$ is non-negative, then the functions 
$$z\mapsto {}_p\Psi_p\Big[_{({\bf b}_q, {\bf B}_q)}^{({\bf a}_p, {\bf A}_p)}\Big|-z\Big],\;\textrm{and}\;\;z\mapsto {}_p\Psi_p\Big[_{({\bf b}_q, {\bf B}_q)}^{({\bf a}_p, {\bf A}_p)}\Big|-z\Big]-\eta e^{-\rho z},$$
are completely monotonic on $(0,\infty).$ 
\end{lemma}




The following Lemma is the so-called  the Chebyshev integral inequality \cite[p. 40]{MM}

\begin{lemma}\label{l11} If $f,g:[a,b]\longrightarrow\mathbb{R}$ are synchoronous (both increasing  or decreasing) integrable functions, and $p:[a,b]\longrightarrow\mathbb{R}$  is a positive integrable function, then 
\begin{equation}\label{OO}
\int_a^b p(t)f(t)dt\int_a^b p(t)g(t)dt\leq \int_a^b p(t)dt\int_a^b p(t)f(t)g(t)dt.
\end{equation}
Note that if $f$ and $g$ are asynchronous (one is decreasing and the other is increasing),
then (\ref{OO}) is reversed. 
 \end{lemma}

\section{positivity of some class of functions related to the Fox H-functions}

Our first main  result is asserted in the following Theorem.

\begin{theorem}\label{T1}Under the conditions
$$(H_1): 0<\tau\leq2, \alpha,\beta\in(0,1),\gamma>0,\;\beta\geq\alpha\gamma,$$
the function 
\begin{equation}\label{e1}
H_{3,2}^{1,2}\left[|\xi|^{-\tau}\Big|^{(1-\frac{d}{2}+\frac{\tau}{2},\frac{\tau}{2}),(2-\gamma,1),(\frac{\tau}{2},\frac{\tau}{2})}_{\;\;\;(1,1),(\alpha-\beta+1,\alpha)}\right],\;\;\xi\in\mathbb{R}^d,
\end{equation}
is non-negative. 
\end{theorem}
\begin{proof}Putting
\begin{equation}
F_{\alpha,\beta}^{\gamma,\tau}(t, x)=\frac{t^{\beta-1}}{\Gamma(\gamma)}H_{1,2}^{1,1}\left[t^\alpha|x|^\tau\Big|^{(1-\gamma,1)}_{(0,1),(1-\beta,\alpha)}\right],\;t>0,\; x\in\mathbb{R}^d.
\end{equation}
By using the fact that $F_{\alpha,\beta}^{\gamma,\tau}$ is radial function in $x$, and since the radial Fourier transform in $d$ dimensions is given in terms of the Hankel transform, that is 
$$\mathcal{F}(f)(|\xi|)=|\xi|^{\frac{2-d}{2}}\int_0^\infty  r^{\frac{d}{2}} J_{\frac{d-2}{2}}(r|\xi|) f(r)dr,$$
where $J_{\frac{d-2}{2}}(.)$ is the Bessel function. In view of the above formula and Lemma \ref{l12}, we get 
\begin{equation}
\begin{split}
\mathcal{F}(F_{\alpha,\beta}^{\gamma,\tau}(t,.))(|\xi|)&=\frac{t^{\beta-1}|\xi|^{\frac{2-d}{2}}}{\Gamma(\gamma)}\int_0^\infty r^{\frac{d}{2}} J_{\frac{d-2}{2}}(r|\xi|)H_{1,2}^{1,1}\left[t^\alpha r^\tau\Big|^{(1-\gamma,1)}_{(0,1),(1-\beta,\alpha)}\right]dr\\
&=\frac{2^{\frac{d}{2}}t^{\beta-1}|\xi|^{\tau-1}}{\Gamma(\gamma)}H_{3,2}^{1,2}\left[t^\alpha \frac{2^\tau}{|\xi|^\tau}\Big|^{(1-\frac{d}{2},\frac{\tau}{2}),(1-\gamma,1),(0,\frac{\tau}{2})}_{\;\;\;(0,1),(1-\beta,\alpha)}\right]\\
&=\frac{2^{\frac{d}{2}-\tau}t^{\beta-\alpha-1}|\xi|^{-1}}{\Gamma(\gamma)}H_{3,2}^{1,2}\left[t^\alpha \frac{2^\tau}{|\xi|^\tau}\Big|^{(1-\frac{d}{2}+\frac{\tau}{2},\frac{\tau}{2}),(2-\gamma,1),(\frac{\tau}{2},\frac{\tau}{2})}_{\;\;\;(1,1),(\alpha-\beta+1,\alpha)}\right].
\end{split}
\end{equation}
In \cite{ZT} the authors  proved that the function $e_{\alpha,\beta}^\gamma(t,\lambda)$ defined by
$$e_{\alpha,\beta}^\gamma(t,\lambda)=t^{\beta-1}E_{\alpha,\beta}^\gamma(-\lambda t^\alpha),$$
is completely monotonic on $(0,\infty),$ under the hypotheses $(H_1).$ Let $0<\tau<2,$ then the function $g(\lambda)=\lambda^{\frac{\tau}{2}}$ is a Bernstein function. Bearing in mind  that the composition of a completely monotone function and a Bernstein function is completely monotone, we conclude that the function $$e_{\alpha,\beta}^\gamma(t,g(\lambda))=t^{\beta-1}E_{\alpha,\beta}^\gamma(-\lambda^{\frac{\tau}{2}} t^\alpha)$$  is completely monotone under the constraint $$\left(0<\tau<2, \alpha,\beta\in(0,1),\gamma>0,\;\beta\geq\alpha\gamma\right).$$ Hence, the function $e_{\alpha,\beta}^\gamma(t,\lambda^{\frac{\tau}{2}})$ is completely monotone under the  hypotheses $(H_1).$ Therefore, by means of Schoenberg Theorem (see \cite[Theorem 7.14]{W}), we earn that the function $$e_{\alpha,\beta}^\gamma(t,|x|^\tau)=F_{\alpha,\beta}^{\gamma,\tau}(t, x)$$ is positive definite on $\mathbb{R}^d$.
Moreover, by using the asymptotic expansion (\ref{asy1}) we obtain
\begin{equation}\label{JKL}
F_{\alpha,\beta}^{\gamma,\tau}(t, x)=\mathcal{O}(1),\;\textrm{as}\;|x|\rightarrow0.
\end{equation}
Further, by means of the asymptotic expansion (\ref{asy2}) we get
\begin{equation}\label{JKL1}
F_{\alpha,\beta}^{\gamma,\tau}(t, x)=\mathcal{O}(x^{-\gamma}),\;\textrm{as}\;|x|\rightarrow\infty.
\end{equation}
Now, collecting (\ref{JKL}) and (\ref{JKL1}) we have
$$F_{\alpha,\beta}^{\gamma,\tau}(t, x)\in L^1(\mathbb{R}^d).$$
In conclusion, the function 
$$x\mapsto F_{\alpha,\beta}^{\gamma,\tau}(t, x),\;t>0$$
is in $L^1(\mathbb{R}^d)$ and positive definite on $\mathbb{R}^d$ under the conditions $(H_1).$ Hence the assumption of Theorem 6.11 in \cite{W} (or Theorem 6.6 in \cite{DERR}) are fulfilled.
However, the function $\mathcal{F}(F_{\alpha,\beta}^{\gamma,\tau}(t,.))(|\xi|)$ is non-negative. The proof of Theorem \ref{T1} is complete.
\end{proof}
\begin{corollary}\label{c1} The following functions:
$$H_{1,2}^{2,0}\left[r\Big|_{(\frac{d}{2}-1,\frac{1}{2}),(\gamma-1,\frac{1}{2})}^{\;\;\;(\beta-\alpha,\frac{\alpha)}{2}}\right],\;\;( \alpha,\beta\in(0,1),\gamma, r>0,\;\beta\geq\alpha\gamma),$$
$$H_{1,2}^{2,0}\left[ r\Big|_{(\frac{d}{2}-\frac{\tau}{2},\frac{1}{2}),(1-\frac{\tau}{2},\frac{1}{2})}^{\;\;\;\;(\beta-\alpha,\frac{\alpha}{\tau})}\right],\;\;( \alpha,\beta\in(0,1),\; \beta\geq\alpha,\; 0<\tau\leq2, d\geq1, \;r>0),$$
are non-negatives.
\end{corollary}
\begin{proof}
Letting $\tau=2$ in (\ref{e1}). By using the Property 3, Property 1 and Property 2 of  Lemma \ref{l1}, we have
\begin{equation*}
\begin{split}
H_{3,2}^{1,2}\left[|\xi|^{-2}\Big|^{(1-\frac{d}{2}+\frac{\tau}{2},\frac{\tau}{2}),(2-\gamma,1),(\frac{\tau}{2},\frac{\tau}{2})}_{\;\;\;(1,1),(\alpha-\beta+1,\alpha)}\right]&=H_{2,1}^{0,2}\left[|\xi|^{-2}\Big|^{(2-\frac{d}{2},1),(2-\gamma,1)}_{\;\;\;(\alpha-\beta+1,\alpha)}\right]\\
&=H_{1,2}^{2,0}\left[|\xi|^2\Big|_{(\frac{d}{2}-1,1),(\gamma-1,1)}^{\;\;\;(\beta-\alpha,\alpha)}\right]\\
&=2^{-1}H_{1,2}^{2,0}\left[|\xi|\Big|_{(\frac{d}{2}-1,\frac{1}{2}),(\gamma-1,\frac{1}{2})}^{\;\;\;(\beta-\alpha,\frac{\alpha)}{2}}\right].
\end{split}
\end{equation*}
Now, setting $\gamma=1$ in (\ref{e1}). In a similar way we earn that the function
$$H_{1,2}^{2,0}\left[ r\Big|_{(\frac{d}{2}-\frac{\tau}{2},\frac{1}{2}),(1-\frac{\tau}{2},\frac{1}{2})}^{\;\;\;\;(\beta-\alpha,\frac{\alpha}{\tau})}\right],\;\;( \alpha,\beta\in(0,1),\; \beta\geq\alpha,\; 0<\tau\leq2, d\geq1, \;r>0,$$
is non-negative. This proves the second statement.
\end{proof}

Upon setting $\gamma=\frac{3}{2}$ and $\tau=1$ respectively, in the first and second assertion of Corollary \ref{c1} and taking the relation 
\begin{equation}\label{LE}
\Gamma(1+s)=\frac{2^s}{\sqrt{\pi}}\Gamma(\frac{1+s}{2})\Gamma(\frac{2+s}{2}),
\end{equation}
into account,  we are led to the following results:

\begin{corollary}The following functions
$$H_{2,2}^{2,0}\left[r\Big|_{(\frac{d}{2}-1,\frac{1}{2}),(1,1)}^{(\beta-\alpha,\frac{\alpha}{2}),(1,\frac{1}{2})}\right],\;\;( \alpha,\beta\in(0,1), r>0,\;2\beta\geq3\alpha),$$
$$H_{2,2}^{2,0}\left[ r\Big|_{(\frac{d}{2}-\frac{1}{2},\frac{1}{2}),(1,1)}^{(\beta-\alpha,\alpha),(1,\frac{1}{2})}\right],\;\;( \alpha,\beta\in(0,1),\; \beta\geq\alpha,\; d\geq1, \;r>0),$$
are non-negatives.
\end{corollary}

\begin{theorem}\label{T2} Under the conditions
$$(H_2):\;\tau\in(0,2], \alpha\in(0,1],\frac{1}{\alpha}-1<\beta,\;\gamma\in\mathbb{R},$$
the function
\begin{equation}\label{e2}
H_{3,2}^{1,2}\left[|\xi|^{-\tau}\Big|^{(1-\frac{d}{2}+\frac{\tau}{2},\frac{\tau}{2}),(1-\frac{\gamma+\beta}{\alpha\beta}),\frac{1}{\alpha\beta}),(\frac{\tau}{2},\frac{\tau}{2})}_{\;\;\;(1,1),(-\frac{\gamma}{\beta},\frac{1}{\alpha})}\right],\;\;\xi\in\mathbb{R}^d,
\end{equation}
is non-negative. 
\end{theorem}
\begin{proof}In \cite{Y1}, Luchko and Kiryakova proved that the function
$${}_1\Psi_1\left[^{(\frac{1+\gamma+\beta}{\alpha\beta},\frac{1}{\alpha\beta})}_{(\frac{1+\gamma+\beta}{\beta},\frac{1}{\beta})}\Big|-z\right]=H_{1,2}^{1,1}\left[z\Big|^{\;\;\;(1-\frac{1+\gamma+\beta}{\alpha\beta},\frac{1}{\alpha\beta})}_{(0,1),(1-\frac{1+\gamma+\beta}{\beta},\frac{1}{\beta})}\right]$$
is completely monotonic on $(0,\infty)$ under the hypotheses $(H_2).$ Therefore, the function 
$$H_{1,2}^{1,1}\left[z^{\frac{\tau}{2}}\Big|^{\;\;\;(1-\frac{1+\gamma+\beta}{\alpha\beta},\frac{1}{\alpha\beta})}_{(0,1),(1-\frac{1+\gamma+\beta}{\beta},\frac{1}{\beta})}\right],$$
is completely monotonic on $(0,\infty),$ under the assumptions of hypotheses $(H_2).$ This implies that the function $G_{\alpha,\beta}^{\gamma,\tau}(z)$ defined by
\begin{equation}
G_{\alpha,\beta}^{\gamma,\tau}(z)=H_{1,2}^{1,1}\left[|z|^\tau\Big|^{\;\;\;(1-\frac{1+\gamma+\beta}{\alpha\beta},\frac{1}{\alpha\beta})}_{(0,1),(1-\frac{1+\gamma+\beta}{\beta},\frac{1}{\beta})}\right],\;\; z\in\mathbb{R}^d.
\end{equation}
is positive definite on $\mathbb{R}^d.$ By repeating the same calculations in the Theorem \ref{T1}, we get
\begin{equation}
\mathcal{F}(G_{\alpha,\beta}^{\gamma,\tau})(|\xi|)=\frac{2^{\frac{d}{2}-\tau}}{|\xi|}H_{3,2}^{1,2}\left[\left(\frac{2}{|\xi|}\right)^\tau\Big|^{(1-\frac{d}{2}+\frac{\tau}{2},\frac{\tau}{2}),(1-\frac{\gamma+\beta}{\alpha\beta}),\frac{1}{\alpha\beta}),(\frac{\tau}{2},\frac{\tau}{2})}_{\;\;\;(1,1),(-\frac{\gamma}{\beta},\frac{1}{\alpha})}\right].
\end{equation}
Finally,  by using  the  fact  that  the  Fourier  transform of positive definite function in $L^1$ is 
non-negative, we deduce that the function $\mathcal{F}(G_{\alpha,\beta}^{\gamma,\tau})(|\xi|)$
is non-negative and this completes the proof of Theorem \ref{T2}.
\end{proof}
\begin{corollary}\label{C2}The following Fox H-functions 
$$H_{1,2}^{2,0}\left[r\Big|_{(\frac{d}{2}-1,\frac{1}{2}),(\frac{\gamma+\beta}{\alpha\beta},\frac{1}{2\alpha\beta})}^{\;\;\;\;(1+\frac{\gamma}{\beta},\frac{1}{2\alpha})}\right],\;\;\left(\alpha\in(0,1],\frac{1}{\alpha}-1<\beta,\;\gamma\in\mathbb{R},\;r>0\right),$$
$$H_{1,2}^{2,0}\left[ r\Big|_{(\frac{d}{2}-\frac{\tau}{2},\frac{1}{2}),(1-\frac{\tau}{2},\frac{1}{2})}^{\;\;\;\;(0,\frac{1}{\alpha\tau})}\right],\;\;\bigg(\tau\in(0,2],\alpha\in(0,1], r>0\bigg),$$
$$H_{1,2}^{2,0}\left[r\Big|^{(0,\frac{1}{\tau})}_{(\frac{d}{2}-\frac{\tau}{2},\frac{1}{2}),\left(1-\frac{\tau}{2},\frac{1}{2\beta}\right)}\right], \left(0<\tau\leq2, \frac{\tau}{2}-1<\beta,\;r>0\right),$$
are non-negative.
\end{corollary}
\begin{proof} Upon setting $\tau=2,\;(\alpha=\frac{1}{\beta},\;\gamma=-\beta)$ and  $(\alpha=-\frac{\gamma}{\beta}=\frac{2}{\tau}),$ respectively in Theorem \ref{T2}. Making use the Property 3 and Property 1 of Lemma \ref{l1}, we thus get
$$H_{1,2}^{2,0}\left[ |\xi|^2\Big|_{(\frac{d}{2}-1,1),(\frac{\gamma+\beta}{\alpha\beta},\frac{1}{\alpha\beta})}^{\;\;\;\;(1+\frac{\gamma}{\beta},\frac{1}{\alpha})}\right],$$
$$H_{1,2}^{2,0}\left[ |\xi|^\tau\Big|_{(\frac{d}{2}-\frac{\tau}{2},\frac{\tau}{2}),(1-\frac{\tau}{2},\frac{\tau}{2})}^{\;\;\;\;(0,\frac{1}{\alpha})}\right],$$
$$H_{1,2}^{2,0}\left[|\xi|^\tau\Big|^{(0,1)}_{(\frac{d}{2}-\frac{\tau}{2},\frac{\tau}{2}),\left(1-\frac{\tau}{2},\frac{\tau}{2\beta}\right)}\right]$$
are non-negative. Finally, taking in account the Property 2 of Lemma \ref{l1} in the above functions we get the desired results.
\end{proof}
\begin{corollary}\label{cc3}The following functions
$$H_{2,2}^{2,0}\left[r\Big|_{(1,1),(\frac{\gamma+\beta}{\alpha\beta},\frac{1}{2\alpha\beta})}^{\;\;\;\;(1+\frac{\gamma}{\beta},\frac{1}{2\alpha}),(1,\frac{1}{2})}\right],\;\;\left(\alpha\in(0,1],\frac{1}{\alpha}-1<\beta,\;\gamma\in\mathbb{R},\;r>0\right),$$
$$H_{2,2}^{2,0}\left[ r\Big|_{(\frac{d}{2}-\frac{1}{2},\frac{1}{2}),(1,1)}^{\;\;\;\;(0,\frac{1}{\alpha}),(1,\frac{1}{2})}\right],\;\;\bigg(\alpha\in(0,1], r>0\bigg),$$
$$H_{2,2}^{2,0}\left[r\Big|^{(0,\frac{1}{d-1}), (1,\frac{1}{2})}_{(1,1),\left(\frac{3}{2}-\frac{d}{2},\frac{1}{2\beta}\right)}\right],\;\left(d\in\left\{2,3\right\}, \frac{d}{2}-\frac{3}{2}<\beta,\;r>0\right),$$
are non-negatives.
\end{corollary}
\begin{proof}Taking $d=3,\;\tau=1$ and $\tau=d-1$ in the first, second and third functions defined in Corollary \ref{C2} and applying the identity (\ref{LE}) we get the desired results. 
\end{proof}
\section{Applications}
\subsection{Some classes of Completely monotonic and positive definite functions related to the H-function} The main focus of this section, is to present some conditions for a class of functions related to the H-function to be completely monotonic and positive definite.

\begin{theorem}\label{T09} Let the parameters range $\alpha,\beta\in(0,1),\gamma>0$ such that $\beta\geq\alpha\gamma,$ then the function 
$$s\mapsto \frac{1}{s}H_{3,3}^{2,2}\left[s\Big|^{(1,\frac{1}{2}),(2-\gamma,1),(\frac{1}{2},\frac{1}{2})}_{\;\;(1,1), (1,1),(1-\beta+\alpha,\alpha)}\right],$$
is completely monotonic on $(0,\infty).$ Furthermore, the function 
$$x\mapsto \left\|x\right\|^{-2}H_{3,3}^{3,0}\left[\left\|x\right\|^2\Big|^{(1,\frac{1}{2}),(2-\gamma,1),(\frac{1}{2},\frac{1}{2})}_{\;\;(1,1), (1,1),(1-\beta+\alpha,\alpha)}\right],$$
is positive definite on $\mathbb{R}^d.$
\end{theorem}
\begin{proof}Specifying $\tau=d=1$ in Theorem \ref{T1} and we used property (1) of Lemma \ref{l1} we deduce that the function
\begin{equation}\label{MIO}
H_{2,3}^{2,1}\left[z\Big|_{(0,\frac{1}{2}),(\gamma-1,1),(\frac{1}{2},\frac{1}{2})}^{\;\;\;(0,1),(\beta-\alpha,\alpha)}\right],
\end{equation}
is non-negative on $(0,\infty).$ In our case $$D=2-\alpha>0,\;\;\textrm{and}\;\;1+\min_{1\leq j\leq 1}\left[\frac{a_j}{A_j}\right]=1>0.$$
Applying the Laplace transform (\ref{laplace}) of the function defined in (\ref{MIO}) we obtain
\begin{equation}
\mathcal{L}\left\{H_{2,3}^{2,1}\left[z\Big|_{(0,\frac{1}{2}),(\gamma-1,1),(\frac{1}{2},\frac{1}{2})}^{\;\;\;(0,1),(\beta-\alpha,\alpha)}\right];s\right\}=\frac{1}{s}H_{3,3}^{2,2}\left[\frac{1}{s}\Big|_{(0,\frac{1}{2}),(\gamma-1,1),(\frac{1}{2},\frac{1}{2})}^{\;\;(0,1), (0,1),(\beta-\alpha,\alpha)}\right].
\end{equation}
Again, in virtue of the formula (1) of Lemma \ref{l1}, the above formula reads as
\begin{equation}
\begin{split}
\mathcal{L}\left\{H_{2,3}^{2,1}\left[z\Big|_{(0,\frac{1}{2}),(\gamma-1,1),(\frac{1}{2},\frac{1}{2})}^{\;\;\;(0,1),(\beta-\alpha,\alpha)}\right];s\right\}&=\frac{1}{s}H_{3,3}^{2,2}\left[s\Big|^{(1,\frac{1}{2}),(2-\gamma,1),(\frac{1}{2},\frac{1}{2})}_{\;\;(1,1), (1,1),(1-\beta+\alpha,\alpha)}\right].
\end{split}
\end{equation}
However, all prerequisites of the Bernstein Characterization Theorem 
for  the  complete  monotone  functions  are  fulfilled, that is, the function
$$s\mapsto \frac{1}{s}H_{3,2}^{2,2}\left[s\Big|^{(1,\frac{1}{2}),(2-\gamma,1),(\frac{1}{2},\frac{1}{2})}_{\;\;(1,1), (1,1),(1-\beta+\alpha,\alpha)}\right],$$
is completely monotonic on $(0,\infty).$ This implies that the function
$$x\mapsto\frac{1}{\left\|x\right\|^2}H_{3,3}^{2,2}\left[\left\|x\right\|^2\Big|^{(1,\frac{1}{2}),(2-\gamma,1),(\frac{1}{2},\frac{1}{2})}_{\;\;(1,1), (1,1),(1-\beta+\alpha,\alpha)}\right],$$
 is positive definite on $\mathbb{R}^d,$ by means of Shoenberg's Theorem, and then the proof of Theorem \ref{T09}, is thus complete.
\end{proof}
\begin{corollary} Let $\nu>-\frac{1}{2}$ and $\rho\in\mathbb{C}$ such that
$\nu+\Re(\rho)+\min(1,1+1/\alpha-\beta/\alpha)>-1$ and $\Re(\rho)+\min(\gamma-1,0)<\frac{3}{2}$ together with another constraints in Theorem \ref{T09}. Then
 the function $$z\mapsto\frac{1}{z^{\rho+\nu}} H_{3,5}^{3,2}\left[2z\Big|_{(\frac{\rho+\nu}{2}, \frac{1}{2}), (0,\frac{1}{2}),(\gamma-1,1),(\frac{1}{2},\frac{1}{2}), (\frac{\rho-\nu}{2}, \frac{1}{2})}^{\;\;\;\;\;\;\;\;(0,1), (0,1),(\beta-\alpha,\alpha)}\right],$$
is positive definite on $\mathbb{R}.$
\end{corollary}
\begin{proof}We can write the formula (\ref{BBBB}) of Lemma \ref{l12} in a form
\begin{equation}\label{sou}
\int_0^\infty x^{\rho+\nu-1} \mathcal{J}_\nu(z x) H_{q,p}^{m,n}\left[x\Big|^{({\bf b}_q, {\bf B}_q)}_{({\bf A}_p,{\bf a}_p)}\right]dx=\frac{\Gamma(\nu+1)2^{\rho+\nu-1}}{z^{\rho+\nu}}H_{q+2,p}^{m,n+1}\left[\frac{2}{z}\Big|^{(1-\frac{\rho+\nu}{2}, \frac{1}{2}),({\bf B}_q,{\bf b}_q),(1-\frac{\rho-\nu}{2}, \frac{1}{2})}_{\;\;\;\;\;\;\;\;\;\;({\bf a}_p, {\bf A}_p)}\right]
\end{equation}
where 
$$\mathcal{J}_\nu(x)=2^\nu\Gamma(\nu+1)\frac{J_\nu(x)}{x^\nu},\;\Re(\nu)>-\frac{1}{2},$$
with $J_\nu(x)$ is the Bessel function of index $\nu.$ However, by using the fact that the function $\mathcal{J}_\nu(x)$ is positive definite function \cite[Proposition 2]{KKKK} and the function 
$$H_{3,3}^{2,2}\left[s\Big|^{(1,\frac{1}{2}),(2-\gamma,1),(\frac{1}{2},\frac{1}{2})}_{\;\;(1,1), (1,1),(1-\beta+\alpha,\alpha)}\right],$$
is non-negative, we obtain that for any finite list of complex numbers $\xi_1,...,\xi_N$ and $z_1,...,z_N\in\mathbb{R},$ 
\begin{equation}
\begin{split}
\sum_{j=1}^N\sum_{k=1}^N&\frac{\xi_j\bar{\xi_k}}{(z_j-z_k)^{(\rho+\nu)}}H_{5,3}^{2,3}\left[\frac{2}{z_j-z_k}\Big|^{(1-\frac{\rho+\nu}{2}, \frac{1}{2}), (1,\frac{1}{2}),(2-\gamma,1),(\frac{1}{2},\frac{1}{2}), (1-\frac{\rho-\nu}{2}, \frac{1}{2})}_{\;\;\;\;\;\;\;\;(1,1), (1,1),(1-\beta+\alpha,\alpha)}\right]=\frac{1}{\Gamma(\nu+1)2^{\rho+\nu-1}}\\
&\times\int_0^\infty r^{\rho+\nu-1} \left[\sum_{j=1}^N\sum_{k=1}^N\xi_j\bar{\xi_k}\mathcal{J}_\nu(r(z_j-z_k))\right]H_{3,3}^{2,2}\left[r\Big|^{(1,\frac{1}{2}),(2-\gamma,1),(\frac{1}{2},\frac{1}{2})}_{\;\;(1,1), (1,1),(1-\beta+\alpha,\alpha)}\right]dr\\
&\geq0,
\end{split}
\end{equation}
which yields that the function
$$z\mapsto\frac{1}{z^{\rho+\nu}} H_{5,3}^{2,3}\left[\frac{2}{z}\Big|^{(1-\frac{\rho+\nu}{2}, \frac{1}{2}), (1,\frac{1}{2}),(2-\gamma,1),(\frac{1}{2},\frac{1}{2}), (1-\frac{\rho-\nu}{2}, \frac{1}{2})}_{\;\;\;\;\;\;\;\;(1,1), (1,1),(1-\beta+\alpha,\alpha)}\right],$$
is positive definite on $\mathbb{R}.$ In virtue of property (1) of Lemma \ref{l1} we deduce that the function 
$$z\mapsto\frac{1}{z^{\rho+\nu}} H_{3,5}^{3,2}\left[2z\Big|_{(\frac{\rho+\nu}{2}, \frac{1}{2}), (0,\frac{1}{2}),(\gamma-1,1),(\frac{1}{2},\frac{1}{2}), (\frac{\rho-\nu}{2}, \frac{1}{2})}^{\;\;\;\;\;\;\;\;(0,1), (0,1),(\beta-\alpha,\alpha)}\right],$$
is positive definite on $\mathbb{R}.$ This is what we intended to show.
\end{proof}
\begin{theorem}\label{T+9}Assume that the parameters $\alpha\in(0,1], \gamma\in\mathbb{R}$ and $\frac{1}{\alpha}-1<\beta.$ Then the function
$$s\mapsto\frac{1}{s}H_{3,3}^{2,2}\left[s\Big|^{(1,\frac{1}{2}),(1-\frac{\gamma+\beta}{\alpha\beta},\frac{1}{\alpha\beta}),(\frac{1}{2},\frac{1}{2})}_{\;\;(1,1), (1,1),(-\frac{\gamma}{\beta},\frac{1}{\alpha})}\right],$$
is completely monotonic on $(0,\infty)$. In addition, the function $$x\mapsto\frac{1}{\left\|x\right\|^2}H_{3,3}^{2,2}\left[\left\|x\right\|^2\Big|^{(1,\frac{1}{2}),(1-\frac{\gamma+\beta}{\alpha\beta},\frac{1}{\alpha\beta}),(\frac{1}{2},\frac{1}{2})}_{\;\;(1,1), (1,1),(-\frac{\gamma}{\beta},\frac{1}{\alpha})}\right],$$
is positive definite on $\mathbb{R}^d$.
\end{theorem}
\begin{proof} We restrict our result in Theorem \ref{T2} to the case $\tau=d=1$ we find that the function
\begin{equation*}
H_{3,2}^{1,2}\left[\frac{1}{r}\Big|^{(1,\frac{1}{2}),(1-\frac{\gamma+\beta}{\alpha\beta},\frac{1}{\alpha\beta}),(\frac{1}{2},\frac{1}{2})}_{\;\;\;(1,1),(-\frac{\gamma}{\beta},\frac{1}{\alpha})}\right],
\end{equation*}
is non-negative on $(0,\infty).$ From the property (1) of Lemma \ref{l1}, we deduce that the function 
\begin{equation*}
H_{2,3}^{2,1}\left[r\Big|_{(0,\frac{1}{2}),(\frac{\gamma+\beta}{\alpha\beta},\frac{1}{\alpha\beta}),(\frac{1}{2},\frac{1}{2})}^{\;\;\;(0,1),(1+\frac{\gamma}{\beta},\frac{1}{\alpha})}\right],
\end{equation*}
is non-negative on $(0,\infty).$ In our case 
$$D=\frac{1}{\alpha\beta}+\frac{1}{\alpha}-1>0,\;\;\textrm{and}\;\;1+\min_{1\leq j\leq 2}\left[\frac{\Re(a_j)}{A_j}\right]=1>0.$$
Thus by Lemma \ref{Laplace} and property (1) of Lemma \ref{l1} we find
\begin{equation}
\begin{split}
\mathcal{L}\left\{H_{2,3}^{2,1}\left[r\Big|_{(0,\frac{1}{2}),(\frac{\gamma+\beta}{\alpha\beta},\frac{1}{\alpha\beta}),(\frac{1}{2},\frac{1}{2})}^{\;\;\;(0,1),(1+\frac{\gamma}{\beta},\frac{1}{\alpha})}\right];s\right\}&=\frac{1}{s}H_{3,3}^{2,2}\left[\frac{1}{s}\Big|_{(0,\frac{1}{2}),(\frac{\gamma+\beta}{\alpha\beta},\frac{1}{\alpha\beta}),(\frac{1}{2},\frac{1}{2})}^{\;\;(0,1), (0,1),(1+\frac{\gamma}{\beta},\frac{1}{\alpha})}\right]\\
&=\frac{1}{s}H_{3,3}^{2,2}\left[s\Big|^{(1,\frac{1}{2}),(1-\frac{\gamma+\beta}{\alpha\beta},\frac{1}{\alpha\beta}),(\frac{1}{2},\frac{1}{2})}_{\;\;(1,1), (1,1),(-\frac{\gamma}{\beta},\frac{1}{\alpha})}\right].
\end{split}
\end{equation}
Consequently, the function
$$s\mapsto\frac{1}{s}H_{3,3}^{2,2}\left[s\Big|^{(1,\frac{1}{2}),(1-\frac{\gamma+\beta}{\alpha\beta},\frac{1}{\alpha\beta}),(\frac{1}{2},\frac{1}{2})}_{\;\;(1,1), (1,1),(-\frac{\gamma}{\beta},\frac{1}{\alpha})}\right],$$
is completely monotonic on $(0,\infty)$ by Bernstein's Theorem, and the function
$$x\mapsto\frac{1}{\left\|x\right\|^2}H_{3,3}^{2,2}\left[\left\|x\right\|^2\Big|^{(1,\frac{1}{2}),(1-\frac{\gamma+\beta}{\alpha\beta},\frac{1}{\alpha\beta}),(\frac{1}{2},\frac{1}{2})}_{\;\;(1,1), (1,1),(-\frac{\gamma}{\beta},\frac{1}{\alpha})}\right],$$
is positive definite on $\mathbb{R}^d$ by virtue of Shoenberg's Theorem.
\end{proof}

\begin{corollary}Let $\nu>-\frac{1}{2}$ and $\rho\in\mathbb{C}$ such that $\Re(\rho)+\nu+\min(1,-(\alpha\gamma)/\beta)>-1$ and $\Re(\rho)+\min(0, \gamma+\beta)<3/2$ together with the constraints in Theorem \ref{T+9}, then the function
$$z\mapsto\frac{1}{z^{\rho+\nu}} H_{3,5}^{3,2}\left[2z\Big|_{(\frac{\rho+\nu}{2}, \frac{1}{2}), (0,\frac{1}{2}),(\frac{\gamma+\beta}{\alpha\beta},\frac{1}{\alpha\beta}),(\frac{1}{2},\frac{1}{2}) ,(\frac{\rho-\nu}{2}, \frac{1}{2})}^{\;\;\;\;\;\;\;\;\;\;(0,1), (0,1),(1+\frac{\gamma}{\beta},\frac{1}{\alpha})}\right],$$
is positive definite on $\mathbb{R}.$
\end{corollary}
\begin{proof}With the aid of the formula (\ref{sou}), we get 
\begin{equation}
\begin{split}
\int_0^\infty x^{\rho+\nu-1} \mathcal{J}_\nu(z x) &H_{3,3}^{2,2}\left[x\Big|^{(1,\frac{1}{2}),(1-\frac{\gamma+\beta}{\alpha\beta},\frac{1}{\alpha\beta}),(\frac{1}{2},\frac{1}{2})}_{\;\;(1,1), (1,1),(-\frac{\gamma}{\beta},\frac{1}{\alpha})}\right]dx=\frac{\Gamma(\nu+1)2^{\rho+\nu-1}}{z^{\rho+\nu}}\\
&\times H_{5,3}^{2,3}\left[\frac{2}{z}\Big|^{(1-\frac{\rho+\nu}{2}, \frac{1}{2}), (1,\frac{1}{2}),(1-\frac{\gamma+\beta}{\alpha\beta},\frac{1}{\alpha\beta}),(\frac{1}{2},\frac{1}{2}) ,(1-\frac{\rho-\nu}{2}, \frac{1}{2})}_{\;\;\;\;\;\;\;\;\;\;(1,1), (1,1),(-\frac{\gamma}{\beta},\frac{1}{\alpha})}\right].
\end{split}
\end{equation}
Further, using the fact that the function $z\mapsto \mathcal{J}_\nu(z)$ is positive definite on $\mathbb{R}$ and the function $$H_{3,3}^{2,2}\left[x\Big|^{(1,\frac{1}{2}),(1-\frac{\gamma+\beta}{\alpha\beta},\frac{1}{\alpha\beta}),(\frac{1}{2},\frac{1}{2})}_{\;\;(1,1), (1,1),(-\frac{\gamma}{\beta},\frac{1}{\alpha})}\right]$$ is non-negative, we deduce that the function 
$$z\mapsto\frac{1}{z^{\rho+\nu}} H_{5,3}^{2,3}\left[\frac{2}{z}\Big|^{(1-\frac{\rho+\nu}{2}, \frac{1}{2}), (1,\frac{1}{2}),(1-\frac{\gamma+\beta}{\alpha\beta},\frac{1}{\alpha\beta}),(\frac{1}{2},\frac{1}{2}) ,(1-\frac{\rho-\nu}{2}, \frac{1}{2})}_{\;\;\;\;\;\;\;\;\;\;(1,1), (1,1),(-\frac{\gamma}{\beta},\frac{1}{\alpha})}\right]$$
is positive definite on $\mathbb{R}$ and consequently the function 
$$z\mapsto\frac{1}{z^{\rho+\nu}} H_{3,5}^{3,2}\left[2z\Big|_{(\frac{\rho+\nu}{2}, \frac{1}{2}), (0,\frac{1}{2}),(\frac{\gamma+\beta}{\alpha\beta},\frac{1}{\alpha\beta}),(\frac{1}{2},\frac{1}{2}) ,(\frac{\rho-\nu}{2}, \frac{1}{2})}^{\;\;\;\;\;\;\;\;\;\;(0,1), (0,1),(1+\frac{\gamma}{\beta},\frac{1}{\alpha})}\right],$$
is positive definite on $\mathbb{R}$, by means of property (1) of Lemma \ref{l1}.
\end{proof}
\subsection{Monotonicity properties for some classes of functions related to the Fox-Wright functions}
The purpose of this section is twofold. First we derive the monotonicity of ratios for some class of functions related to the Fox-Wright functions. Second, we give sufficient conditions for some functions involving the Fox-Wright functions to be completely monotonic.
 The following Theorem are powerful tools to treat the monotonicity
of ratios between two Fox-Wright functions.

\begin{theorem}\label{T4}Let $\psi:[a,b]\longrightarrow(0,\infty)$ be a twice differentiable mapping on $[a,b],$ with $0\leq a<b,$ such that the function $t\mapsto t \psi^\prime(t)/ \psi(t)$ is increasing ( decreasing) on $[a,b].$ We define the function $\mathbb{K}_{q,p}^{n,m}[.]$ by
$$\mathbb{K}_{q,p}^{n,m}\left[_{(\textbf{b}_q, \textbf{B}_q)}^{(\textbf{a}_p,\textbf{A}_p)}\Big|\sigma,\delta;z\right]=\frac{\int_a^b t^{-1}H_{q,p}^{n,m}\left[_{(\textbf{b}_q+\delta\textbf{B}_q, \textbf{B}_q)}^{(\textbf{a}_p+\delta\textbf{A}_p,\textbf{A}_p)}\Big|t\right]\left[\psi(zt)\right]^\sigma dt}{\int_a^b t^{-1}H_{q,p}^{n,m}\left[_{(\textbf{b}_q, \textbf{B}_q)}^{(\textbf{a}_p,\textbf{A}_p)}\Big|t\right]\left[\psi(zt)\right]^\sigma dt},\;z,\delta>0,\;\sigma\in\mathbb{R}-\left\{0\right\}.$$
Assume that the function $H_{q,p}^{n,m}[.]$ is non-negative. Then the function $\mathbb{K}_{q,p}^{n,m}[z]$ is increasing (decreasing) if $\sigma>0$ and decreasing (increasing) if $\sigma<0.$
 \end{theorem}
\begin{proof}By means of Property (4) of Lemma \ref{l1}, we can write the function $\mathbb{K}_{q,p}^{n,m}[z]$ in the following form:
\begin{equation}
\mathbb{K}_{q,p}^{n,m}\left[_{(\textbf{b}_q, \textbf{B}_q)}^{(\textbf{a}_p,\textbf{A}_p)}\Big|\sigma,\delta;z\right]=\frac{\int_a^b t^{\delta-1}H_{q,p}^{n,m}\left[_{(\textbf{b}_q, \textbf{B}_q)}^{(\textbf{a}_p,\textbf{A}_p)}\Big|t\right]\left[\psi(zt)\right]^\sigma dt}{\int_a^b t^{-1}H_{q,p}^{n,m}\left[_{(\textbf{b}_q, \textbf{B}_q)}^{(\textbf{a}_p,\textbf{A}_p)}\Big|t\right]\left[\psi(zt)\right]^\sigma dt}.
\end{equation}
Therefore,
$$\left[\int_a^b t^{-1}H_{q,p}^{n,m}\left[_{(\textbf{b}_q, \textbf{B}_q)}^{(\textbf{a}_p,\textbf{A}_p)}\Big|t\right]\left[\psi(zt)\right]^\sigma dt\right]^2 \frac{\partial}{\partial z}\mathbb{K}_{q,p}^{n,m}\left[_{(\textbf{b}_q, \textbf{B}_q)}^{(\textbf{a}_p,\textbf{A}_p)}\Big|\sigma,\delta;z\right]$$
\begin{equation}\label{malouka}
\begin{split}
&=\sigma\left(\int_a^b t^{\delta} H_{q,p}^{n,m}\left[_{(\textbf{b}_q, \textbf{B}_q)}^{(\textbf{a}_p,\textbf{A}_p)}\Big|t\right]\psi^\prime(zt)\left[\psi(zt)\right]^{\sigma-1} dt\right)\left(\int_a^b t^{-1} H_{q,p}^{n,m}\left[_{(\textbf{b}_q, \textbf{B}_q)}^{(\textbf{a}_p,\textbf{A}_p)}\Big|t\right]\left[\psi(zt)\right]^\sigma dt\right)\\
&-\sigma\left(\int_a^b t^{\delta-1} H_{q,p}^{n,m}\left[_{(\textbf{b}_q, \textbf{B}_q)}^{(\textbf{a}_p,\textbf{A}_p)}\Big|t\right]\left[\psi(zt)\right]^\sigma dt\right)\left(\int_a^b  H_{q,p}^{n,m}\left[_{(\textbf{b}_q, \textbf{B}_q)}^{(\textbf{a}_p,\textbf{A}_p)}\Big|t\right]\psi^\prime(zt)\left[\psi(zt)\right]^{\sigma-1} dt\right).
\end{split}
\end{equation}
Putting
\begin{eqnarray*}
p(t)&=&t^{-1} [\psi(tz)]^\sigma H_{q,p}^{n,m}\left[_{(\textbf{b}_q, \textbf{B}_q)}^{(\textbf{a}_p,\textbf{A}_p)}\Big|t\right],\\
f(t)&=&t^{\delta},\;\;g(t)=t\psi^\prime(zt)/\psi(zt).
\end{eqnarray*}
 In the case when the function $g$ is increasing, the function $f$ and $g$ are synchoronous. Thus, by Lemma \ref{l11}, we obtain
$$\left(\int_a^b t^{-1} H_{q,p}^{n,m}\left[_{(\textbf{b}_q, \textbf{B}_q)}^{(\textbf{a}_p,\textbf{A}_p)}\Big|t\right]\left[\psi(zt)\right]^\sigma dt\right)\left(\int_a^b t^{\delta} H_{q,p}^{n,m}\left[_{(\textbf{b}_q, \textbf{B}_q)}^{(\textbf{a}_p,\textbf{A}_p)}\Big|t\right]\psi^\prime(zt)\left[\psi(zt)\right]^{\sigma-1} dt\right)$$
\begin{eqnarray}\label{Malouka}
&\geq& \left(\int_a^b t^{\delta-1} H_{q,p}^{n,m}\left[_{(\textbf{b}_q, \textbf{B}_q)}^{(\textbf{a}_p,\textbf{A}_p)}\Big|t\right]\left[\psi(zt)\right]^\sigma dt\right)\left(\int_a^b  H_{q,p}^{n,m}\left[_{(\textbf{b}_q, \textbf{B}_q)}^{(\textbf{a}_p,\textbf{A}_p)}\Big|t\right]\psi^\prime(zt)\left[\psi(zt)\right]^{\sigma-1} dt\right).
\end{eqnarray}
Then, keeping (\ref{malouka}) and (\ref{Malouka}) in mind, we deduce that the function $\mathbb{K}_{q,p}^{n,m}[z]$ is increasing if $\sigma>0$ and decreasing if $\sigma<0.$ Moreover, if the function $g$ is decreasing then the inequality (\ref{Malouka}) is reversed and consequently he function $\mathbb{K}_{q,p}^{n,m}[z]$ is decreasing if $\sigma>0$ and increasing if $\sigma<0.$  This completes the proof of Theorem \ref{T4}.
\end{proof}
\begin{remark} We note that in the case when $\delta<0$ and the function $t\mapsto t\psi^\prime(t)/\psi(t)$ is decreasing, we obtain the same monotonicity property of $\mathbb{K}_{q,p}^{n,m}[z]$ as in Theorem \ref{T4}.
\end{remark}
\begin{theorem}\label{T5}Let $\delta>0.$ Under the conditions
$$ (H_3): 0<a_1\leq...\leq a_p, 0<b_1\leq...\leq b_p,\;\sum_{j=1}^k b_j-\sum_{j=1}^k a_j\geq0, k=1,...,p.$$
Then the ratios:
$$z\mapsto {}_{p+1}\Psi_{p}\left[^{(\sigma,1),(\textbf{a}_p+\delta A,A)}_{\;\;\;(\textbf{b}_p+\delta A,A)}\Big|-z\right]\Big/{}_{p+1}\Psi_{p}\left[^{(\sigma,1),(\textbf{a}_p,A)}_{\;\;\;({\bf b}_p,A)}\Big|-z\right],\;\sigma>0$$
is decreasing  on $(0,1).$
\end{theorem}
\begin{proof}In \cite[Remark 2]{KMJMAA}, the author proved that the function  Fox H-function $H_{p,p}^{p,0}[_{(\textbf{a}_p,A)}^{(\textbf{b}_p,A)}|t]$ is non-negative. In our case,  $\psi(t)=(1+t)^{-1}$ and the function $t\mapsto t\psi^\prime(t)/\psi(t)$ is decreasing on $(-1,\infty).$ However, bearing in mind the tools of Theorem \ref{T4} and Lemma \ref{l3}, we derive the desired assertions asserted by Theorem \ref{T5}.
\end{proof}
\begin{example}The following function 
$$z\mapsto\frac{\varphi^\tau(b+\delta \tau,c+\delta \tau,-z)}{\varphi^\tau(b,c,-z)},\;(c>b>0,\;\tau,\delta>0,\;|z|<1$$
is decreasing  on $(0,1)$ where $\varphi^\tau$ is the $\tau-$Kummer hypergeometric,  defined by \cite{V1} 
$$\varphi^\tau(b,c,z)=\sum_{k=0}^\infty \frac{\Gamma(b+k\tau)}{\Gamma(c+k\tau)}\frac{z^k}{k!},\;(c>b>0,\;\tau>0,\;|z|<1.)$$
\end{example}

\begin{theorem}\label{T6} Assume that $\mu,\delta>0,\;\sum_{i=1}^p A_i=\sum_{j=1}^q B_j$ and $\gamma\geq1.$ If the H-function $H_{q,p}^{p,0}[.]$ is non-negative, then the function 
$$z\mapsto {}_{p+1}\Psi_{q}\left[^{(\sigma+\delta,1),(\textbf{a}_p+\delta {\bf A}_p,\textbf{A}_p)}_{(\textbf{b}_p+\delta {\bf B}_q,\textbf{B}_q)}\Big|z\right]\Big/{}_{p+1}\Psi_{q}\left[^{(\textbf{a}_p,{\bf A}_p)}_{({\bf b}_q,{\bf B}_q)}\Big|z\right],$$
is decreasing on $(0,1).$
\end{theorem}
\begin{proof}Keeping Theorem \ref{T4} and Lemma \ref{l3} straightforward calculations complete the proof.
\end{proof}
\begin{corollary}\label{c79} Let $\delta,\sigma>0.$ Assume that
$$(H_4): \tau\in(0,1), d-\tau\geq1, \beta> \frac{d}{2}+\frac{1}{2}.$$
The function 
$$z\mapsto{}_{3}\Psi_{1}\left[^{(\sigma+\delta,1),(\frac{d}{2}-\frac{\tau}{2}+\frac{\delta}{2},\frac{1}{2}),(1-\frac{\tau}{2}+\frac{\delta}{2},\frac{1}{2})}_{\;\;\;\;\;\;\;\;\;(\beta-\tau+\delta,1)}\Big|-z\right]\Big/{}_{3}\Psi_{1}\left[^{(\sigma,1),(\frac{d}{2}-\frac{\tau}{2},\frac{1}{2}),(1-\frac{\tau}{2},\frac{1}{2})}_{\;\;\;\;\;\;\;\;\;(\beta-\tau,1)}\Big|-z\right],$$
is decreasing on $(0,1).$ Moreover, the function 
$$z\mapsto{}_{3}\Psi_{1}\left[^{(\sigma,1),(\frac{d}{2}-\frac{\tau}{2},\frac{1}{2}),(1-\frac{\tau}{2},\frac{1}{2})}_{\;\;\;\;\;\;\;\;\;(\beta-\tau,1)}\Big|-z\right],$$
is logarithmically completely monotonic on $(0,1).$
\end{corollary}
\begin{proof} We consider the second function defined in Corollary \ref{c1}. However, the function 
$$H_{1,2}^{2,0}\left[r\Big|_{(\frac{d}{2}-\frac{\tau}{2},\frac{1}{2}),(1-\frac{\tau}{2},\frac{1}{2})}^{\;\;\;\;\;\;\;\;\;(\beta-\tau,1)}\right],$$ is non-negative. So, in this case the hypotheses of Theorem \ref{T6} is equivalent to the hypotheses $(H_4).$ Now, applying Theorem \ref{T6} and Lemma \ref{l3} respectively, we deduce that the above assertions hold true. This ends the proof.
\end{proof}
\begin{corollary}\label{ccc5} Let $\sigma, \delta>0.$ In addition assume that the hypotheses
\begin{displaymath}
(H_5): \left\{ \begin{array}{ll}
\gamma\in\mathbb{R}, \alpha\in \big[\frac{3-\sqrt{5}}{2},1\big),\\
\min\left(d-2, 2\left(\gamma+\frac{1}{1-\alpha}\right)\right)\geq1,\\
\frac{5}{2}>\frac{1+\gamma(1-\alpha)^2}{\alpha}+\frac{d}{2}
\end{array} \right.
\end{displaymath}
holds true. Then the function 
$$z\mapsto{}_{3}\Psi_{1}\left[^{(\sigma+\delta,1),(\frac{d}{2}+\frac{\delta}{2}-1,\frac{1}{2}),(\frac{1+\gamma(1-\alpha)}{\alpha}+\frac{\delta(1-\alpha)}{2\alpha},\frac{1-\alpha}{2\alpha})}_{\;\;\;\;\;\;\;\;\;(1+\gamma(1-\alpha)+\frac{\delta}{2\alpha},\frac{1}{2\alpha})}\Big|-z\right]\Big/{}_{3}\Psi_{1}\left[^{(\sigma,1),(\frac{d}{2}-1,\frac{1}{2}),(\frac{1+\gamma(1-\alpha)}{\alpha},\frac{1-\alpha}{2\alpha})}_{\;\;\;\;\;\;\;\;\;(1+\gamma(1-\alpha),\frac{1}{2\alpha})}\Big|-z\right],$$
is decreasing on $(0,1)$.
 Moreover, the function $$z\mapsto{}_{3}\Psi_{1}\left[^{(\sigma,1),(\frac{d}{2}-1,\frac{1}{2}),(\frac{1+\gamma(1-\alpha)}{\alpha},\frac{1-\alpha}{2\alpha})}_{\;\;\;\;\;\;\;\;\;(1+\gamma(1-\alpha),\frac{1}{2\alpha})}\Big|-z\right]$$
is logarithmically completely monotonic on $(0,1).$
\end{corollary}
\begin{proof} Again, by using Corollary \ref{C2} , we have that the function
$$H_{1,2}^{2,0}\left[ r\Big|_{(\frac{d}{2}-1,\frac{1}{2}),(\frac{\gamma+\beta}{\alpha\beta},\frac{1}{2\alpha\beta})}^{\;\;\;\;(1+\frac{\gamma}{\beta},\frac{1}{2\alpha})}\right]\;r>0,$$
is non-negative. In our case, the hypotheses of Theorem \ref{T6} is equivalent to the hypotheses of $(H_5)$. Applying  Theorem \ref{T6} and Lemma \ref{l3} leads to the desired results.
\end{proof}

Obviously, by repeating the same calculations as above with Theorem \ref{T6}, Corollary \ref{C2} (third function) and Lemma \ref{l3}, we find the following result:

\begin{corollary}\label{c7} Let $\sigma, \delta>0.$  Suppose also  that 
\begin{displaymath}
(H_6): \left\{ \begin{array}{ll}
\tau\in(0,2], \frac{\tau}{2}-1<\beta,  \frac{1}{\tau}=\frac{1}{2}+\frac{1}{2\beta},\\
\min(2\beta-\beta\tau,d-\tau)\geq1\\
\tau>\frac{d}{2}+\frac{1}{2},
\end{array} \right.
\end{displaymath}
Then, the function
$$z\mapsto{}_{3}\Psi_{1}\left[^{(\sigma+\delta,1),(\frac{d}{2}-\frac{\tau}{2}+\frac{\delta}{2},\frac{1}{2}),(1-\frac{\tau}{2}+\frac{\delta}{2\beta},\frac{1}{2\beta})}_{\;\;\;\;\;\;\;\;\;\;\;\;\;(\frac{\delta}{\tau},\frac{1}{\tau})}\Big|-z\right]\Big/{}_{3}\Psi_{1}\left[^{(\sigma,1),(\frac{d}{2}-\frac{\tau}{2},\frac{1}{2}),(1-\frac{\tau}{2},\frac{1}{2\beta})}_{\;\;\;\;\;\;\;\;\;\;\;\;\;(0,\frac{1}{\tau})}\Big|-z\right]$$
is decreasing on $(0,1).$ Furthermore, the function
$$z\mapsto{}_{3}\Psi_{1}\left[^{(\sigma,1),(\frac{d}{2}-\frac{\tau}{2},\frac{1}{2}),(1-\frac{\tau}{2},\frac{1}{2\beta})}_{\;\;\;\;\;\;\;\;\;(0,\frac{1}{\tau})}\Big|-z\right]$$
is logarithmically completely monotonic on $(0,1).$ 
\end{corollary}

By repeating the procedure of the proofs of the above Corollary and  make use Theorem \ref{T6}, Corollary \ref{cc3} (first function)  and Lemma \ref{l3},  leads us to the asserted results in  Corollary \ref{c77}.

\begin{corollary}\label{c77} Let $\delta,\sigma>0.$ Under the conditions
\begin{displaymath}
(H_7): \left\{ \begin{array}{ll}
\alpha\in(0,1],\;\gamma\in\mathbb{R}, 2(\gamma+\beta)\geq1\\\;\;\;\;\; \frac{1}{\alpha}=1+\frac{1}{\alpha\beta}, \frac{1}{\alpha}-1<\beta,\\
1+\frac{\gamma}{\beta}\left(1-\frac{1}{\alpha}\right)=1+\frac{1}{\alpha\beta}>\frac{1}{\alpha},
\end{array} \right.
\end{displaymath}
the function 
$$z\mapsto {}_3\Psi_2\left[^{(\sigma+\delta,1),(1,1),(\frac{\gamma+\beta}{\alpha\beta}+\frac{\delta}{2\alpha\beta},\frac{1}{2\alpha\beta})}_{\;\;\;\;(1+\frac{\gamma}{\beta}+\frac{\delta}{2\alpha},\frac{1}{2\alpha}),(1+\frac{\delta}{2},\frac{1}{2})}\Big|-z\right]\Big/{}_3\Psi_2\left[^{(\sigma,1),(1,1),(\frac{\gamma+\beta}{\alpha\beta},\frac{1}{2\alpha\beta})}_{\;\;\;\;(1+\frac{\gamma}{\beta},\frac{1}{2\alpha}),(1,\frac{1}{2})}\Big|-z\right],$$
is decreasing on $(0,1).$ Moreover, the function 
$$z\mapsto {}_3\Psi_2\left[^{(\sigma,1),(1,1),(\frac{\gamma+\beta}{\alpha\beta},\frac{1}{2\alpha\beta})}_{\;\;\;\;(1+\frac{\gamma}{\beta},\frac{1}{2\alpha}),(1,\frac{1}{2})}\Big|-z\right],$$
is logarithmically completely monotonic on $(0,1).$
\end{corollary}

\begin{theorem}The following assertions are true:\\
\noindent 1. The function 
 $$z\mapsto{}_{2}\Psi_{1}\left[^{(\frac{d}{2}-\frac{\tau}{2},\frac{1}{2}),(1-\frac{\tau}{2},\frac{1}{2})}_{\;\;\;\;\;\;\;\;\;(\beta-\tau,1)}\Big|-z\right],$$
is completely monotonic on $(0,\infty)$ under the hypotheses of corollary \ref{c79}.\\
\noindent 2. The function 
$$z\mapsto{}_{2}\Psi_{1}\left[^{(\frac{d}{2}-1,\frac{1}{\tau}),(\frac{1+\gamma(1-\alpha)}{\alpha},\frac{1-\alpha}{\alpha\tau})}_{\;\;\;\;\;\;\;\;\;(1+\gamma(1-\alpha),\frac{1}{\alpha\tau})}\Big|-z\right],$$
is completely monotonic on $(0,\infty),$ under the hypotheses of corollary \ref{ccc5}.\\
\noindent 3. The function
$$z\mapsto{}_{2}\Psi_{1}\left[^{(\frac{d}{2}-\frac{\tau}{2},\frac{1}{2}),(1-\frac{\tau}{2},\frac{1}{2\beta})}_{\;\;\;\;\;\;\;\;\;(0,\frac{1}{\tau})}\Big|-z\right]$$
is completely monotonic on $(0,\infty),$ under the hypotheses of corollary \ref{c7}.\\
\noindent 4. The function
$$z\mapsto {}_2\Psi_2\left[^{(1,1),(\frac{\gamma+\beta}{\alpha\beta},\frac{1}{2\alpha\beta})}_{(1+\frac{\gamma}{\beta},\frac{1}{2\alpha}),(1,\frac{1}{2})}\Big|-z\right],$$
is completely monotonic on $(0,\infty),$ under the hypotheses of corollary \ref{c77}.
\end{theorem}
\begin{proof} The above assertions follows immediately by combining Lemma \ref{T71} with Corollary \ref{c1}, Corollary \ref{C2} and Corollary \ref{cc3} respectively, under  some restrictions on the  parameters of the Fox H-functions  which allow us to conclude that it is non-negative.
\end{proof}
\begin{theorem} Letting 
$\eta_1=\sqrt{\pi} 2^{\tau-\frac{d}{2}+\frac{1}{2}}.$
 Assume that
$$(H_8): \tau\in(0,1), d-\tau\geq1.$$
Then, the functions 
$$g_{1}:=z\mapsto{}_{2}\Psi_{1}\left[^{(\frac{d}{2}-\frac{\tau}{2},\frac{1}{2}),(1-\frac{\tau}{2},\frac{1}{2})}_{\;\;\;\;\;\;\;\;\;(\frac{d}{2}+\frac{1}{2}-\tau,1)}\Big|-z\right],\;\textrm{and}\;g_{2}:=z\mapsto{}_{2}\Psi_{1}\left[^{(\frac{d}{2}-\frac{\tau}{2},\frac{1}{2}),(1-\frac{\tau}{2},\frac{1}{2})}_{\;\;\;\;\;\;\;\;\;(\frac{d}{2}+\frac{1}{2}-\tau,1)}\Big|-z\right]-\eta_1 e^{-2 z},$$
are  completely monotonic on $(0,\infty).$
\end{theorem}
\begin{proof} We consider the second function defined in Corollary \ref{c1}. In particular, the function 
$$H_{1,2}^{2,0}\left[r\Big|^{(\frac{d}{2}-\frac{\tau}{2},\frac{1}{2}),(1-\frac{\tau}{2},\frac{1}{2})}_{\;\;\;\;\;\;\;(\frac{d}{2}+\frac{1}{2}-\tau,1)}\right],$$ is non-negative. Here, the  parameters of the above function and the hypotheses $(H_8)$ satisfies the hypotheses of Lemma \ref{l4}. Now, applying  Lemma \ref{l4} we obtain that the function $g_1$ and $g_2$ are completely monotonic on $(0,\infty).$
\end{proof}

\begin{theorem}Let 
$$\eta_2=\sqrt{2\pi}\left(\frac{1}{2}\right)^{\frac{d}{2}-\frac{3}{2}}\left(\frac{1-\alpha}{2\alpha}\right)^{\frac{2+2\gamma(1-\alpha)-\alpha}{2\alpha}}\left(\frac{1}{2\alpha}\right)^{\gamma(\alpha-1)-\frac{1}{2}},\;\rho_2=\sqrt{2}\left(\frac{1}{2\alpha}\right)^{\frac{1}{2\alpha}}\left(\frac{1-\alpha}{2\alpha}\right)^{\frac{\alpha-1}{2\alpha}}.$$
Assume that the hypotheses
\begin{displaymath}
(H_9): \left\{ \begin{array}{ll}\alpha\in [\frac{3-\sqrt{5}}{2},1), \gamma\in\mathbb{R},\\
\;\min\left(d-2, 2\left(\gamma+\frac{1}{1-\alpha}\right)\right)\geq1,\\
\;\frac{1+\gamma(1-\alpha)^2}{\alpha}+\frac{d}{2}=\frac{5}{2},
\end{array} \right.
\end{displaymath}
holds true. Then, the functions
$$g_3:=z\mapsto{}_{2}\Psi_{1}\left[^{(\frac{d}{2}-1,\frac{1}{2}),(\frac{1+\gamma(1-\alpha)}{\alpha},\frac{1-\alpha}{2\alpha})}_{\;\;\;\;\;\;\;\;\;(1+\gamma(1-\alpha),\frac{1}{2\alpha})}\Big|-z\right],\;\;g_4:=z\mapsto{}_{2}\Psi_{1}\left[^{(\frac{d}{2}-1,\frac{1}{2}),(\frac{1+\gamma(1-\alpha)}{\alpha},\frac{1-\alpha}{2\alpha})}_{\;\;\;\;\;\;\;\;\;(1+\gamma(1-\alpha),\frac{1}{2\alpha})}\Big|-z\right]-\eta_2 e^{-\rho_2 z},$$
are completely monotonic on $(0,\infty).$
\end{theorem}
\begin{proof}By using the Corollary \ref{C2}, the function
$$H_{1,2}^{2,0}\left[ r\Big|_{(\frac{d}{2}-1,\frac{1}{2}),(\frac{1+\gamma(1-\alpha)}{\alpha},\frac{1-\alpha}{2\alpha})}^{\;\;\;\;(1+\gamma(1-\alpha),\frac{1}{2\alpha})}\right]\;r>0,$$ 
is non-negative and your parameters with hypotheses $(H_9)$ satisfies the statements of Lemma \ref{l4}, which yields that the function $g_3$ and $g_4$ are completely monotonic on $(0,\infty).$
\end{proof}

\begin{theorem} Let
$$\eta_3=\sqrt{2}\left(\frac{1}{2\alpha}\right)^{-\frac{1}{2}-\frac{\gamma}{\beta}}\left(\frac{1}{2\alpha\beta}\right)^{\frac{\gamma+\beta}{\alpha\beta}-\frac{1}{2}},\rho_3=\left(\frac{1}{2}\right)^{\frac{1}{2}}\left(\frac{1}{2\alpha}\right)^{\frac{1}{2\alpha}}\left(\frac{1}{2\alpha\beta}\right)^{\frac{1}{2\alpha\beta}},$$
such that the following hypotheses 
\begin{displaymath}
(H_{10}): \left\{ \begin{array}{ll}
\alpha\in(0,1], \frac{1}{\alpha}-1<\beta,\;\gamma\in\mathbb{R}\\\;\;\;\;\; 2(\gamma+\beta)\geq1,\\
1+\frac{\gamma}{\beta}\left(1-\frac{1}{\alpha}\right)=1+\frac{1}{\alpha\beta}=\frac{1}{\alpha},
\end{array} \right.
\end{displaymath}
hold true. Then the functions
$$g_5:=z\mapsto {}_2\Psi_2\left[^{(1,1),(\frac{\gamma+\beta}{\alpha\beta},\frac{1}{2\alpha\beta})}_{\;\;\;\;(1+\frac{\gamma}{\beta},\frac{1}{2\alpha}),(1,\frac{1}{2})}\Big|-z\right],\;g_6:=z\mapsto {}_2\Psi_2\left[^{(1,1),(\frac{\gamma+\beta}{\alpha\beta},\frac{1}{2\alpha\beta})}_{\;\;\;\;(1+\frac{\gamma}{\beta},\frac{1}{2\alpha}),(1,\frac{1}{2})}\Big|-z\right]-\eta_3 e^{-\rho_3 z},$$
are completely monotonic on $(0,\infty).$
\end{theorem}
\begin{proof}Collecting Lemma \ref{l4} and Corollary \ref{cc3} (the first function) we confirm the stated assertion.
\end{proof}

\end{document}